\documentclass[11pt,a4paper,reqno]{article} 

\usepackage{amssymb,amsmath}
\usepackage{amsthm}
\usepackage{latexsym}
\usepackage[colorlinks=true,linktocpage=true,pagebackref=false, citecolor=black,linkcolor=black]{hyperref}
\usepackage{graphics}
\usepackage{euscript}
\usepackage{mathrsfs}
\usepackage[dvips]{curves}
\usepackage{tikz}
\usepackage{pgf}
\usetikzlibrary{arrows,decorations.pathmorphing,backgrounds,positioning,fit}
\usepackage{pdfsync}
\usepackage{comment}

\newtheorem{thm}{Theorem}[section]

\newtheorem{defn}[thm]{Definition}
\newtheorem{lem}[thm]{Lemma}
\newtheorem{cor}[thm]{Corollary}

 \newcommand{\R}{{\mathbb R}}

\newcommand{\p}{{\EuScript P}}
\newcommand{\q}{{\EuScript Q}}
\newcommand{\Ss}{{\EuScript S}}
\newcommand{\Tt}{{\EuScript T}}

\newcommand{\Rr}{{\EuScript R}}

\newcommand{\x}{x} \newcommand{\y}{y}
 \renewcommand{\t}{t}

\newcommand{\bfq}{q}

\numberwithin{equation}{section}

\begin{document}
\begin{center}
\Large
Unbounded convex polygons as polynomial images of the plane\footnote{This article is based on a part of the doctoral dissertation of the author, written under the supervision of his advisor Prof. Jose F. Fernando.}
\end{center}

\begin{flushright}
Carlos Ueno\footnote{The author is an external collaborator of Spanish MTM2011-22435.} \\
IES La Vega de San Jos\'e \\
35015 Las Palmas de Gran Canaria, SPAIN\\
\verb+cuenjac@gmail.com+
\end{flushright}

\section*{Abstract}
{\small In this note we show that unbounded convex polygons with nonparallel unbounded edges are polynomial images of $\R^2$, whereas their interiors are polynomial images of $\R^3$.}

\section{Introduction}

The study of polynomial images of $\R^n$ has attracted the attention of this author during recent years. Fernando and Gamboa can be considered pioneers in this field; they found some properties that a set on the plane must satisfy in order to be the image set of a polynomial map $f:\R^2\to\R^2$ (\cite{fg1}), as well as some nontrivial examples. Their work is partly based on Jelonek's contribution to the understanding of the set of points where a map is not proper (\cite{j1,j2}). In fact, those results are crucial to prove that the set of infinite points of a polynomial image is connected (see \cite{fu2} for further details). We refer the reader to \cite{fg1,fgu1,fu} for a careful study of potential applications and for a complete picture of the state of the art.

Since finding a full characterization of such polynomial images seems at present a difficult problem, part of our attention has been directed to finding ample families of semialgebraic sets (mainly with piecewise linear boundaries) that can actually be represented `constructively' as images of polynomial maps. Thus, in \cite{u2} we prove that complements of convex polygons and their interiors (with just one exception, that of the region bounded by two parallel lines) are polynomial images of $\R^2$. We also know that the only `open' convex polygons that are polynomial images are open halfplanes and open angles (see \cite[Cor. 3.7]{fg2}). Related results that refer to more general settings can be found in \cite{fgu1,fu}.

In this note we complete our study of the relationship between convex polygons and images of polynomial maps:

\begin{thm}\label{thm:pleg} 
Let $\p\subset\R^2$ be either a half-plane or an unbounded convex polygon with nonparallel unbounded edges. Then $\p$ is the image of a polynomial map $f:\R^2\to\R^2$. 
\end{thm}
\begin{cor}\label{cor:pleg}
Let $\p\subset\R^2$ be an unbounded convex polygon with nonparallel unbounded edges. Then ${\rm Int}(\p)$ is the image of a polynomial map $f:\R^3\to\R^2$.
\end{cor}
As commented above, the previous corollary can only be improved to a polynomial map $f:\R^2\to\R^2$ when $\p$ is either and open half-plane or an open angle. Thus, the previous results provide the full picture concerning the representation of unbounded convex polygons and their interiors as polynomial images.

To simplify our terminology, we reserve the term \em {\tt V}-polygon \em for a polygon which satisfies the properties stated in Theorem~\ref{thm:pleg} (we get a {\tt V} shape by extending its unbounded edges till they meet). As it is already known, bounded polygons or convex polygons with parallel unbounded edges, together with their interiors, cannot be polynomial images of the plane \cite[Remark 1.3]{fg1}. 

Recall that a map $f:\R^2\to\R^2$ is \em proper \em when $f^{-1}(K)$ is a compact set whenever $K$ itself is compact. Polynomial maps are not always proper maps, and simple examples of this fact are constant maps or linear projections. Proper maps are nice in the sense that they prevent the collapse of points at infinity. Our proof of Theorem~\ref{thm:pleg} involves polynomial maps which are not proper, and an interesting challenge is the following: \em Can any {\tt V}-polygon be represented as the image of a proper polynomial map $f:\R^2\to\R^2$?\em

\section{Preliminaries and main tool}

We first establish the basic notions and notations that we need. We denote the segment or edge joining points $p,q\in\R^2$ by $pq$. If we are given a point $p$ and a vector $\vec{v}$ of the plane, the (closed) ray with origin $p$ and direction $\vec{v}$ is denoted by $p\vec{v}$. If an unbounded convex polygon $\p$ with $n$ edges has consecutive vertices $p_1$,\dots, $p_{n-1}$, with unbounded edges given by rays $p_1\vec{v}$, $p_2\vec{w}$, then we write $\p$ as  $[\vec{v},p_1,\dots,p_{n-1},\vec{w}]$. We also use the fact that a convex polygon $\p$ can be represented as a finite intersection of halfplanes given by one-degree inequalities $\ell_i(x,y)\ge 0$, where the equations $\ell_i=0$ correspond to the lines containing the edges of $\p$ (for more on convex polygons and polyhedra see \cite[Ch.12]{ber2}). We use the notations $[a,b]$ and $]a,b[$ for respectively closed and open intervals in $\R$.

Although the following notions extend naturally to higher dimensions, here we restrict our attention to the 2-dimensional setting. A \em polynomial map \em is a map $f=(f_1,f_2):\R^2\to\R^2$ where $f_1$ and $f_2$ are (real) polynomial functions on $\R^2$. A \em polynomial image \em is the image of a polynomial map. A \em basic semialgebraic set \em $\Ss$ in $\R^2$ is a set that can be described as follows:
\[
\Ss:=\{(x,y)\in\R^2: g_1(x,y)*0,\dots,g_k(x,y)*0\},
\]
where the $g_i\in\R[x,y]$ and each $*\in\{>,\ge\}$. The finite collection $\mathfrak G_\Ss:=\{g_1,\dots,g_k\}$ will be called a \em defining family for $\Ss$\em. Of course, a defining family $\mathfrak G_\Ss$ for $\Ss$ is not unique. 

It is also the case that, given a point $p$ in the boundary $\partial \Ss$ of $\Ss$, we must have $g_i(p)=0$ for some polynomial $g_i\in\mathfrak G_\Ss$. A relevant family of basic semialgebraic sets are convex polygons, which constitute our main interest in this paper.
\begin{defn}
A \emph{curtain} is a basic semialgebraic set $\Ss$ such that:
\begin{enumerate}
\item[{\rm(i)}] $\Ss\subset\{y\ge 0\}$.
\item[{\rm(ii)}] $\Ss_r:=\Ss\cap\{x=r\}$ is either the emptyset or a (closed or open) ray for each $r\in\R$.\end{enumerate}
\end{defn}
Notice that each ray $\Ss_r$ in a curtain $\Ss$ has the form $\{r\}\times\langle s_r,+\infty[$ where $s_r\ge 0$ and the symbol $\langle$ represents either $]$ or $[$. It is also the case that, after a suitable affine change of coordinates, every $\tt V$-polygon $\p$ can be assumed to be a curtain. We also have

\begin{lem}\label{fupol}
If a $\tt V$-polygon $\p$ lying on the upper half-plane $\{y\ge 0\}$ has unbounded edges $p_1\vec{v}$, $p_{n-1}\vec{w}$ where $\vec{v}=(\alpha_1,\beta_1)$ and $\vec{w}=(\alpha_2,\beta_2)$ satisfy $\alpha_1\le0\le \alpha_2$, then $\p$ is a curtain.
\end{lem}
Lemma \ref{fupol} remains valid for any set $\p'$ which is obtained by deleting some vertices and/or edges of the $\tt V$-polygon $\p$. 

We will carry out the proof of Theorem~\ref{thm:pleg} by using induction on the number of edges of the polygon. We first state a result which constitutes our main tool to ``fold'' plane regions in a convenient way, by means of polynomial maps. Along our exposition, we frequently substitute a polygon by an affinely equivalent one, more suitable to our purposes; this will not harm our arguments. 

\begin{lem}\label{thm:dob} 
Let $\Ss$ be a curtain with defining family $\mathfrak G_\Ss$, together with numbers $-\infty\le c<d\le +\infty$. Let $\pi:\R^2\to\R$ be the projection $(x,y)\mapsto x$ and $\mathfrak G'_\Ss$ the subset of polynomials $\mathfrak G_\Ss$ which are not divisible by $(x-r)$ for each $r\in\ ]c,d[$. Consider the polynomials 
$$
\varphi(\x,\y)=\prod_{g_i\in\mathfrak G'_\Ss}g_i(\x,\y)\quad\text{and}\quad h(\x,\y)=\y(1+\psi(\x)\varphi(\x,\y))^2\in\R[\x,\y],
$$
where
$$
\psi(\x)=\begin{cases}
(\x-c)(\x-d),&\text{if $c,d\in\R$}\\
(\x-d), &\text{if $c={-\infty}$}\\
(c-\x), &\text{if $d=+\infty$}.
\end{cases}
$$
Then the map $f:\R^2\to\R^2,\, (x,y)\mapsto(x,h(x,y))$ satisfies
$$
f(\Ss_r)=\begin{cases}
\Ss_r, &\text{if $r\in\pi(\Ss)\setminus\ ]c,d[$}\\
\{r\}\times[0, +\infty[, &\text{if $r\in\ ]c,d[\ \cap\ \pi(\Ss)$}.
\end{cases}
$$
In particular, $f(\Ss)=\Ss\cup((]c,d[\ \cap\ \pi(\Ss)) \times[0,+\infty[)$. 
\end{lem}
\begin{proof} Since $\Ss=\bigcup_{r\in\pi(\Ss)}\Ss_r$, we have $f(\Ss)=\bigcup_{r\in\pi(\Ss)}f(\Ss_r)$. Recall that we write $\Ss_r:=\{r\}\times\langle s_r,+\infty[$ where $s_r\geq0$ and $\langle$ represents either $]$ or $[$, depending on $\Ss_r$ being an open or closed ray respectively. It is clear that $f(\Ss_r)\subset\{r\}\times[0,+\infty[$. For each $r\in\pi(\Ss)$ there is a polynomial $g_i\in\mathfrak G'_\Ss$ such that $g_i(r,s_r)=0$, and therefore $\varphi(r,s_r)=0$. We consider now the polynomials
$$
\gamma_r(\t):=1+\psi(r)\varphi(r,\t)\ \text{ and }\ h_r(\t):=h(r,\t)=\gamma_r^2(\t)\t,
$$
and realize that, since $\varphi(r,s_r)=0$, we have $\gamma_r(s_r)=1$ and $h_r(s_r)=s_r$ for $r\in\pi(\Ss)$. We distinguish two cases:
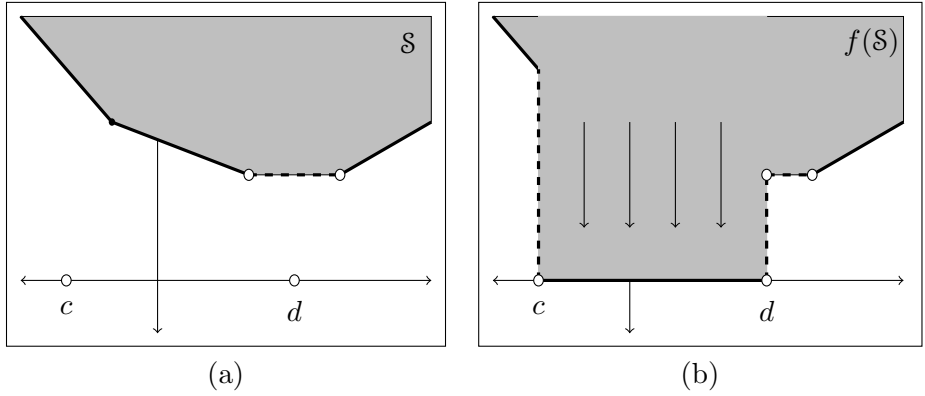
\begin{figure}[ht]
\begin{center}
\scalebox{1}{
\begin{tabular}{cc}
\begin{tikzpicture}[yscale=0.7,xscale=0.6,framed]
\draw[<->] (-3,0)--(6,0);
\draw[<->] (0,-1)--(0,5);
\draw[fill=lightgray, very thin, opacity=1] (-3,5)--(-1,3)--(2,2)--(4,2)--(6,3)--(6,5)--cycle;
\draw[very thick] (-3,5)--(-1,3)--(2,2)(4,2)--(6,3);
\draw[very thick,dashed] (2,2)--(4,2);
\fill [black,opacity=1] (-1,3) circle (2pt);
\fill [black,opacity=1] (2,2) circle (2pt);
\fill [black,opacity=1] (4,2) circle (2pt);
\path (5.5,4.5) node{$\Ss$};
\fill [white,draw=black,opacity=1] (2,2) circle (3pt);
\fill [white,draw=black,opacity=1] (4,2) circle (3pt);
\fill [white,draw=black,opacity=1] (-2,0) circle (3pt);
\fill [white,draw=black,opacity=1] (3,0) circle (3pt);
\path (-2,-0.2) node[below]{$c$} (3,-0.2) node[below]{$d$};
\end{tikzpicture}&
\begin{tikzpicture}[yscale=0.7,xscale=0.6,framed]
\draw[<->] (-3,0)--(6,0);
\draw[<->] (0,-1)--(0,5);
\draw[fill=lightgray, very thin, opacity=1] (-3,5)--(-1,3)--(2,2)--(4,2)--(6,3)--(6,5)--cycle;
\draw[very thick] (-3,5)--(-1,3)--(2,2)(4,2)--(6,3);
\draw[very thick, dashed] (3,2)--(4,2);
\fill [black,opacity=1] (-1,3) circle (2pt);
\fill [black,opacity=1] (2,2) circle (2pt);
\fill [black,opacity=1] (4,2) circle (2pt);
\draw[fill=lightgray, color=lightgray, opacity=1] (-2,5)--(-2,0)--(3,0)--(3,5)--cycle;
\draw[very thick,dashed] (-2,0)--(-2,4)(3,0)--(3,2);
\draw[very thick] (-2,0)--(3,0);
\draw[thin,->](-1,3)--(-1,1);
\draw[thin,->](0,3)--(0,1);
\draw[thin,->](1,3)--(1,1);
\draw[thin,->](2,3)--(2,1);
\fill [white,draw=black,opacity=1] (3,2) circle (3pt);
\fill [white,draw=black,opacity=1] (4,2) circle (3pt);
\fill [white,draw=black,opacity=1] (-2,0) circle (3pt);
\fill [white,draw=black,opacity=1] (3,0) circle (3pt);
\path (-2,-0.2) node[below]{$c$} (3,-0.2) node[below]{$d$};
\path (5.3,4.5) node{$f(\Ss)$};
\end{tikzpicture}\\
(a)&(b)
\end{tabular}}
\caption{{\small Effect of the map $f$ on a basic semialgebraic set $\Ss$.}}\label{fig:ori1}
\end{center}
\end{figure}

\vskip 0.2cm
\noindent{\sc Case 1}. If $r\in\pi(\Ss)\setminus\,]c,d[$, then $h_r(\t)$ is a polynomial of odd degree and $h_r(t)\ge t$ for each $t\in\langle s_r,+\infty[$. This is so because in this case $\psi(r)\geq0$, and from $(r,t)\in\Ss$ we obtain $g_i(r,t)\geq0$ for each $g_i\in\mathfrak G_\Ss$, so that also $\varphi(r,t)\geq0$. Therefore,
\[
h_r(t)=t(1+\psi(r)\varphi(r,t))^2\ge t.
\]
Hence, we must have $f(\Ss_r)=\Ss_r$, since $\lim_{t\to+\infty}h_r(t)=+\infty$ and $h_r(s_r)=s_r$.
\vskip 0.2cm
\noindent{\sc Case 2}. If $r\in\pi(\Ss)\,\cap\,]c,d[$, the polynomial $\gamma_r(\t)=1+\psi(r)\varphi(r,\t)$ is not constant. Otherwise $\varphi_r(t):=\varphi(r,\t)$ is also constant, and this implies that either some $g_i(r,\t)$ is zero or each $g_j(r,\t)$ is a nonzero constant polynomial. The first option takes place if $(\x-r)$ divides $g_i(\x,\y)$, and this cannot happen whereas the second implies that $\Ss_r=\R$, and this is ruled out in our hypothesis.

Now, in this case the main coefficient of $\gamma_r(\t)$ is negative because $\psi(r)<0$ and the main coefficient of $\varphi_r(\t)$ is positive, since each restriction $g_i|_{\Ss_r}$ is nonnegative. Therefore $\lim_{t\to+\infty}\gamma_r(t)={-\infty}$ and, since $\gamma_r(s_r)=1$, there is $t_0\in\ ]s_r,+\infty[$ with $\gamma_r(t_0)=0$ and consequently $h_r(t_0)=0$. But we also have $h_r(t)\geq0$ for each $t\in\langle s_r,+\infty[$ and $\lim_{t\to+\infty}h_r(t)=+\infty$, so we conclude that $f(\Ss_r)=\{r\}\times[0,+\infty[$.
\end{proof}

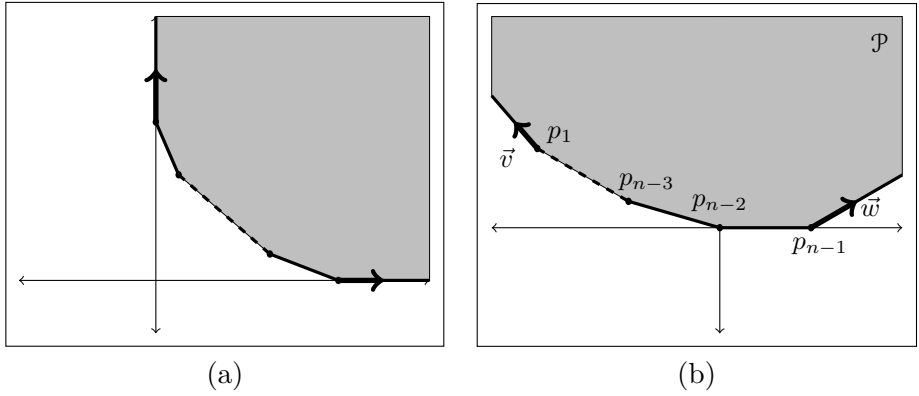
\begin{figure}[t]
\begin{center}
\scalebox{1}{
\begin{tabular}{cc}
\begin{tikzpicture}[yscale=0.7,xscale=0.6,framed]
\draw[<->] (-3,0)--(6,0);
\draw[<->] (0,-1)--(0,5);
\draw[fill=lightgray, very thin] (0,5)--(0,3)--(0.5,2)--(2.5,0.5)--(4,0)--(6,0)--(6,5)--cycle;
\draw[very thick] (0,5)--(0,3)--(0.5,2) (2.5,0.5)--(4,0)--(6,0);
\draw[very thick,dashed] (0.5,2) -- (2.5,0.5);
\fill [black,opacity=1] (4,0) circle (2pt);
\fill [black,opacity=1] (0.5,2) circle (2pt);
\fill [black,opacity=1] (2.5,0.5) circle (2pt);
\fill [black,opacity=1] (0,3) circle (2pt);
\draw[->,line width=2pt] (4,0)-- (5,0);
\draw[->,line width=2pt] (0,3)-- (0,4);
\end{tikzpicture}&
\begin{tikzpicture}[yscale=0.7,xscale=0.6,framed]
\draw[<->] (-5,0)--(4,0);
\draw[<->] (0,-2)--(0,4);
\draw[fill=lightgray, very thin] (-5,2.5)--(-4,1.5)--(-2,0.5)--(0,0)--(2,0)--(4,1)--(4,4)--(-5,4)--cycle;
\draw[very thick] (-5,2.5)--(-4,1.5)(-2,0.5)--(0,0)--(2,0)--(4,1);
\draw[very thick,dashed] (-4,1.5)--(-2,0.5);
\fill [black,opacity=1] (-4,1.5) circle (2pt);
\fill [black,opacity=1] (-2,0.5) circle (2pt);
\fill [black,opacity=1] (0,0) circle (2pt);
\fill [black,opacity=1] (2,0) circle (2pt);
\draw[->,line width=2pt] (2,0)-- (3,0.5);
\draw[->,line width=2pt] (-4,1.5)--node[below left]{\small $\vec{v}$} (-4.5,2);
\path (-4,1.4) node[above right]{\small $p_1$}-- (-1.6, 0.5) node[above] {\small $p_{n-3}$}-- (0,0.1) node[above] {\small $p_{n-2}$}--(2.2,0) node[below]{\small $p_{n-1}$} (3.3,0.4) node{\small $\vec{w}$};
\path (3.5,3.5) node{\small $\p$};
\end{tikzpicture}\\
(a)&(b)
\end{tabular}}
\caption{{\small Setting up the initial position of our {\tt V}-polygon $\p$.}}\label{fig:ori1}
\end{center}
\end{figure}

\section{Proofs of Theorem~\ref{thm:pleg} and Corollary~\ref{cor:pleg}}

\begin{proof}[Proof of Theorem~\em\ref{thm:pleg}]
As we already mentioned we use induction on the number $n$ of edges of the {\tt V}-polygon $\p$. We conduct the proof in a sequence of steps:

\vspace{1mm}
\noindent{\bf Step 1.} \em Initial steps of the inductive procedure\em. {\tt V}-polygons with just one edge are in fact closed halfplanes in $\R^2$ and we are reduced to show that $\p:=\{y\ge 0\}$ is a polynomial image; take for instance $(x,y)\mapsto(x,y^2)$. For $n=2$ it is enough to show that the closed quadrant $\p:=\{x\ge0,\ y\ge0\}$ is a polynomial map; take for instance $(x,y)\mapsto(x^2,y^2)$. 

\vspace{1mm}
\noindent{\bf Step 2.} \em Preparation of the inductive step\em. Let us consider now a {\tt V}-polygon $\p:=[\vec{v},p_1,\dots,p_{n-1},\vec{w}]$ with $n>2$ edges, which has at least two vertices. After an affine change of coordinates we can assume that $\p$ is a curtain and:
\begin{enumerate}
\item[(i)] If $p_i=(a_i,b_i)$ then $a_1<\cdots<a_{n-2}=0<a_{n-1}$ and $b_{n-2}=b_{n-1}=0$.
\item[(ii)] If $\vec{v}=(\alpha_1,\beta_1)$ and $\vec{w}=(\alpha_2,\beta_2)$, then $\alpha_1<0<\alpha_2$ and $\beta_1,\beta_2>0$.
\end{enumerate}
To that end, we apply an affine change of coordinates that places $\p$ in the first quadrant, so that its unbounded edges are contained in the axes, and then apply an isometry that takes the vertex $p_{n-2}$ into the origin and the edge $p_{n-2}p_{n-1}$ on the axis $\{y=0\}$ in such a way that $\p\subset\{y\ge 0\}$. Note that, when situated in this form, none of the edges of $\p$ is contained in a vertical line (see Figure~\ref{fig:ori1}(b)). 

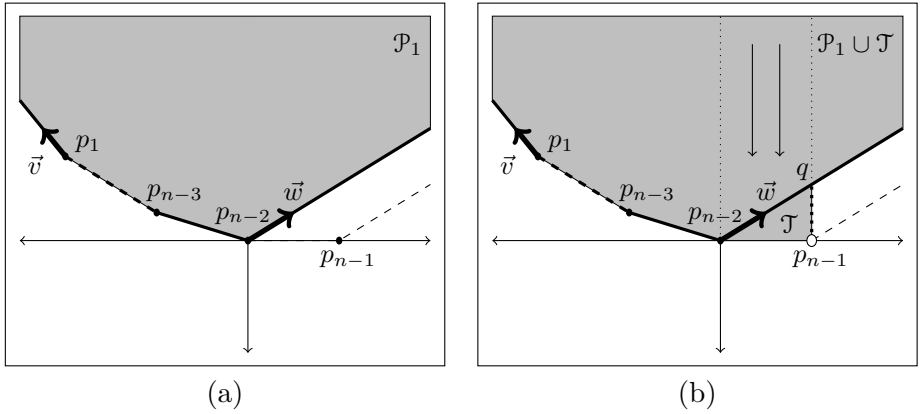
\begin{figure}[h]
\begin{center}
\scalebox{1}{
\begin{tabular}{cc}
\begin{tikzpicture}[yscale=0.744,xscale=0.6,framed]
\draw[<->] (-5,0)--(4,0);
\draw[<->] (0,-2)--(0,4);
\draw[fill=lightgray, very thin] (-5,2.5)--(-4,1.5)--(-2,0.5)--(0,0)--(4,2)--(4,4)--(-5,4)--cycle;
\draw[very thick] (-5,2.5)--(-4,1.5)(-2,0.5)--(0,0);
\draw[very thick,dashed] (-4,1.5)--(-2,0.5);
\draw[very thick] (0,0)--(4,2);
\draw[->,line width=2pt] (0,0)-- (1,0.5)node[above]{\small $\vec{w}$};
\fill [black,opacity=1] (-4,1.5) circle (2pt);
\fill [black,opacity=1] (-2,0.5) circle (2pt);
\fill [black,opacity=1] (0,0) circle (2pt);
\fill [black,opacity=1] (2,0) circle (2pt);
\draw[dashed] (0,0)--(2,0)--(4,1);
\draw[->,line width=2pt] (-4,1.5)--node[below left]{\small $\vec{v}$} (-4.5,2);
\path (-4,1.4) node[above right]{\small $p_1$}-- (-1.6, 0.5) node[above] {\small $p_{n-3}$}-- (-0.1,0.1) node[above] {\small $p_{n-2}$}--(2.2,0) node[below]{\small $p_{n-1}$};
\path (3.5,3.5) node{\small $\p_1$};
\end{tikzpicture}
&
\begin{tikzpicture}[yscale=0.744,xscale=0.6,framed]
\draw[<->] (-5,0)--(4,0);
\draw[<->] (0,-2)--(0,4);
\draw[fill=lightgray,very thin] (-5,2.5)--(-4,1.5)--(-2,0.5)--(0,0)--(2,0)--(2,1)--(4,2)--(4,4)--(-5,4)--cycle;
\draw[very thick] (-5,2.5)--(-4,1.5)(-2,0.5)--(0,0)--(4,2);
\draw[very thick,dashed] (-4,1.5)--(-2,0.5);
\draw[very thick,dotted] (2,0)--(2,1);
\draw[thin,dotted] (0,0)--(0,4) (2,1)--(2,4);
\draw[dashed] (2,0)--(4,1);
\fill [black,opacity=1] (-4,1.5) circle (2pt);
\fill [black,opacity=1] (-2,0.5) circle (2pt);
\fill [black,opacity=1] (0,0) circle (2pt);
\fill [white,opacity=1,draw=black] (2,0) circle (3pt);
\draw[draw=black,->](0.7,3.5)--(0.7,1.5);
\draw[draw=black, ->](1.3,3.5)--(1.3,1.5);
\draw[->,line width=2pt] (0,0)-- (1,0.5)node[above]{\small $\vec{w}$};
\draw[->,line width=2pt] (-4,1.5)--node[below left]{\small $\vec{v}$} (-4.5,2);
\path (-4,1.4) node[above right]{\small $p_1$}-- (-1.6, 0.5) node[above] {\small $p_{n-3}$}-- (-0.1,0.1) node[above] {\small $p_{n-2}$}--(2.2,0) node[below]{\small $p_{n-1}$}--(1.8,0.9)node[above]{\small $\bfq$}--(1.5,0.3)node{\small $\Tt$} (1.9,0.6) node[right]{\small $$};
\path (2.95,3.5) node{\small $\p_1\cup \Tt$};
\end{tikzpicture}\\
(a) & (b)
\end{tabular}}
\caption{{\small Passing from $\p_1\subset\p$ to $\p_1\cup \Tt$}}\label{fig:ori2}
\end{center}
\end{figure}

\noindent{\bf Step 3.} \em Immediate consequences of the inductive step\em. Let us consider now the {\tt V}-polygon $\p_1=[\vec{v},p_1,\dots, p_{n-2},\vec{w}]$, which has one edge less than $\p$ (see Figure~\ref{fig:ori2}(a)). By the induction hypothesis there is a polynomial map 
\begin{equation}\label{1}
f_0:\R^2\to\R^2\quad\text{with}\quad f_0(\R^2)=\p_1.
\end{equation}
Let $\ell_1,\ldots,\ell_{n-1}\in\R[\x,\y]$ be one-degree polynomials whose zeroes are the lines containing the edges of $\p_1$, and satisfying $\p_1=\{\ell_1\geq0,\ldots,\ell_{n-1}\geq0\}$; none of these is multiple of some $\x-r$ since no edge of $\p_1$ is contained in a vertical line. After setting $\Ss:=\p_1$, $\mathfrak G_\Ss=\mathfrak G'_\Ss=\{\ell_1,\dots,\ell_{n-1}\}$, $c:=a_{n-2}=0$, $d:=a_{n-1}$ and $\varphi(\x,\y):=\ell(x,y)=\ell_1(\x,\y)\cdots\ell_{n-1}(\x,\y)$, we apply Lemma~\ref{thm:dob} to obtain the polynomial map
$$
f_1:\R^2\to\R^2,\, (x,y)\mapsto(x,(1+\ell(x,y) x (x-a_{n-1}))^2y),
$$
which maps $\p_1$ onto the region
\begin{equation}\label{2}
\q:=f_1(\p_1)=\p_1\cup (]0,a_{n-1}[\times[0,+\infty[)=\p_1\cup \Tt\subset \p.
\end{equation}
Here $\Tt$ is the nonclosed triangle $(]0,a_{n-1}[\times[0,+\infty[)\setminus \p_1$, whose vertices are $p_{n-2}=(0,0)$, $p_{n-1}$ and $q:=(a_0,b_0)$ (see Figure~\ref{fig:ori2}(b)).

\vspace{1mm}
\noindent{\bf Step 4.} \em Covering the missing part of the polygon off a vertex\em. We now relocate the polygon $\p$ using an affine change of coordinates $\tau_1$ to get
\begin{equation}\label{3}
 \p':=\tau_1(\p)=[\vec{v}',p'_1,\dots,p'_{n-1},\vec{w}'],\quad\q':=\tau_1(\q),
\end{equation}
and $\p'_1:=\tau_1(\p_1)$, in such a way that the edge $p'_{n-1}\vec{w}'$ is contained in $\{y=0\}$, the vertex $q'=\tau_1(q)=(a'_0,b'_0)$ has both coordinates positive, $p'_{n-1}=(0,0)$ and the vector $\vec{v}'=(\alpha'_1,\beta'_1)$ satisfies $\alpha'_1< 0$ (as shown in Figure~\ref{fig:ori3}(a)).

Now the ray $p'_{n-2}\vec{w}'$ in $\p'_1$ is parallel to $\{y=0\}$, and the vector $\vec{u}':=q'-p'_{n-1}=(a'_0,b'_0)$ has positive coordinates. We set
$$
\p'_2:=[\vec{v}',p_1',\dots,p'_{n-1},\vec{u}']\quad\text{and}\quad\Rr:=\{x\geq a'_0,\ y\geq b'_0\},
$$ 
and notice that $\q'=\big(\p'_2\setminus p'_{n-1}\vec{u}'\big)\cup\Rr$. Notice that $\p'_2$ is a curtain.

Let $\ell'_1,\ldots,\ell_n'\in\R[\x,\y]$ denote the one-degree polynomials whose zeroes contain the edges of $\p'_2$, and such that $\p'_2=\{\ell'_1\geq0,\ldots,\ell_n'\geq0\}$. None of these polynomials is multiple of some $\x-r$ because $\p'_2$ does not contain vertical edges. Apply now Lemma~\ref{thm:dob} with $\Ss:=\p'_2\setminus p'_{n-1}\vec{u}'$, $\mathfrak G_\Ss=\mathfrak G'_\Ss=\{\ell'_1,\dots,\ell'_n\}$, $c:=0$, $d:=+\infty$ and $\varphi(x,y):=\ell'(x,y)=\ell_1'(\x,\y)\cdots\ell_n'(\x,\y)$, so that we obtain the polynomial map
$$
f_2:\R^2\to\R^2,\, (x,y)\mapsto(x,(1+\ell'(x,y)x)^2y),
$$
which produces the polynomial images 
$$
f_2(\p'_2\setminus p'_{n-1}\vec{v}')=\p'\setminus\{p'_{n-1}\}\quad\text{and}\quad
f_2(\Rr)\subset\ ]a_0,+\infty[\times[0,+\infty[\ \subset\p'\setminus\{p'_{n-1}\};
$$ 
hence, (see Figure~\ref{fig:ori3}(b))
\begin{equation}\label{4}
f_2(\q')=\p'\setminus\{p'_{n-1}\}.
\end{equation} 

\vspace{1mm}
\noindent{\bf Step 5.} \em Covering the missing vertex of the polygon\em. Since $\p$ and $\p'$ are affinely equivalent it is enough to show that this last {\tt V}-polygon is a polynomial image. To that end, we find a polynomial map that maps $\p'\setminus\{p'_{n-1}\}$ onto $\p'$, covering the `missing' vertex $p'_{n-1}=(0,0)$.

To achieve this, we relocate $\p'$ using an affine change of coordinates $\tau_2$ to obtain the convex polyhedron
\begin{equation}\label{5}
\p'':=\tau_2(\p')=[\vec{v}'',p''_1,\dots,p''_{n-1},\vec{w}''],
\end{equation}
which satisfies the following conditions (Figure~\ref{fig:ori4}(a)): 
\begin{enumerate}
\item[(i)] The vertex $p''_{n-1}$ is located at the origin.
\item[(ii)] The bounded edge $p''_{n-2}p''_{n-1}$ is contained in $\{y=0\}$.
\item[(iii)] The unbounded edge $p''_{n-1}\vec{w}''$ lies on the line $\{x=0\}$.
\item[(iv)] The polygon $\p''$ is contained in the second quadrant $\{x\le 0, y\ge 0\}$.
\end{enumerate}

\begin{figure}[ht]
\begin{center}
\scalebox{1}{
\begin{tabular}{cc}
\begin{tikzpicture}[scale=0.6,framed]
\draw[<->] (0,-1)--(0,6);
\draw[<->] (-5,0)--(4,0);
\draw[fill=lightgray, very thin] (-5,6)--(-4.5,4)--(-3.5,2.5)--(-2,1)--(0,0)--(0.5,1)--(4,1)--(4,6)--cycle;
\draw[very thick] (-5,6)--(-4.5,4)(-3.5,2.5)--(-2,1)--(0,0)(0.5,1)--(4,1);
\draw[very thick,dashed] (-4.5,4)--(-3.5,2.5);
\draw[->,very thick,dotted] (0,0)--(0.5,1);
\draw[very thin, dashed] (-2,1)--(0.5,1);
\draw[thin,dotted] (0.5,1)--(0.5,6);
\draw[thin,dotted] (0.5,1)--(3,6);
\fill [black,opacity=1] (-4.5,4) circle (2pt);
\fill [black,opacity=1] (-3.5,2.5) circle (2pt);
\fill [black,opacity=1] (-2,1) circle (2pt);
\fill [black,opacity=1] (0.5,1) circle (2pt);
\fill [white,opacity=1,draw=black] (0,0) circle (4pt);
\draw[->,line width=2pt] (0.5,1)--node[above right]{\small $\vec{w}'$} (1.5,1);
\draw[->,line width=2pt] (-4.5,4)--node[above right]{\small $\vec{v}'$} (-4.75,5);
\path (-4.5,4.1) node[right]{\small $p'_1$}-- (-3.6, 2.8) node[right] {\small $p'_{n-3}$}-- (-2.2,1.4) node[right] {\small $p'_{n-2}$}--(0.1,0) node[below left]{\small $p'_{n-1}$}--(0.5,1)node[above]{\small $q'$}--(0.6,0.4) node{$\vec{u}'$};
\path (3.5,5.5) node{\small $\q'$};
\end{tikzpicture}
&
\begin{tikzpicture}[scale=0.6,framed]
\draw[<->] (0,-1)--(0,6);
\draw[<->] (-5,0)--(4,0);
\draw[fill=lightgray, very thin] (-5,6)--(-4.5,4)--(-3.5,2.5)--(-2,1)--(0,0)--(4,0)--(4,6)--cycle;
\draw[very thick] (-5,6)--(-4.5,4)(-3.5,2.5)--(-2,1)--(0,0)--(4,0);
\draw[very thick,dashed] (-4.5,4)--(-3.5,2.5);
\draw[thin,dotted] (0,0)--(0,6);
\fill [black,opacity=1] (-4.5,4) circle (2pt);
\fill [black,opacity=1] (-3.5,2.5) circle (2pt);
\fill [black,opacity=1] (-2,1) circle (2pt);
\draw[draw=black,->](1,4.5)--(1,1.5);
\draw[draw=black, ->](2,4.5)--(2,1.5);
\draw[draw=black, ->](3,4.5)--(3,1.5);
\draw[->,line width=2pt] (0,0)--node[above right]{\small $\vec{w}'$} (1,0);
\fill [white,opacity=1,draw=black] (0,0) circle (4pt);
\draw[->,line width=2pt] (-4.5,4)--node[above right]{\small $\vec{v}'$} (-4.75,5);
\path (-4.5,4.1) node[right]{\small $p'_1$}-- (-3.6, 2.8) node[right] {\small $p'_{n-3}$}-- (-2.1,1.3) node[right] {\small $p'_{n-2}$}--(0.1,0) node[below left]{\small $p'_{n-1}$};
\path (3.5,5.5) node{$\p'$};
\end{tikzpicture}\\
(a) & (b)
\end{tabular}}
\caption{{\small Obtaining $\p'\setminus\{ p_{n-1}\}$ as a polynomial image.}}\label{fig:ori3}
\end{center}
\end{figure}
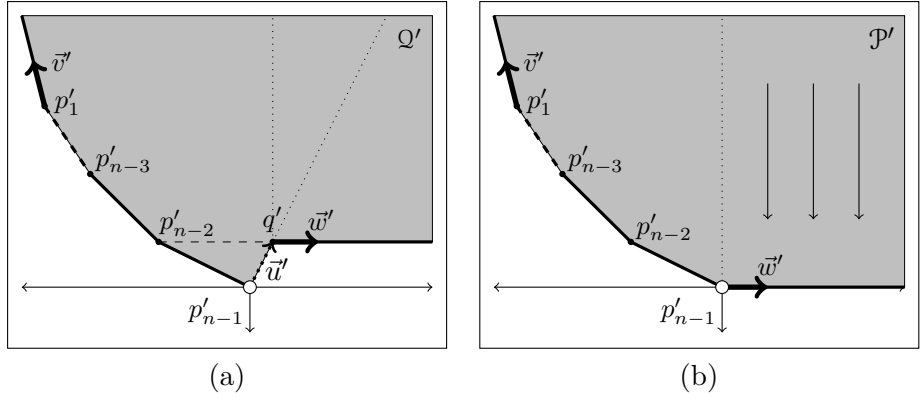
We denote now by $\ell''_1,\ldots,\ell''_n\in\R[\x,\y]$ the one-degree polynomials whose zeroes contain the edges of $\p''$ and satisfy $\p''=\{\ell''_1\geq0,\ldots,\ell''_n\geq0\}$ where $\ell''_{n-1}=\y$ and $\ell''_{n}=-\x$. None of $\ell''_1,\ldots,\ell''_{n-1}$ is multiple of some $\x-r$, since $\ell''_n$ provides the only vertical edge of $\p''$. 

\begin{figure}[ht]
\begin{center}
\scalebox{1}{
\begin{tabular}{cc}
\begin{tikzpicture}[scale=0.75,framed]
\draw[<->] (-6,0)--(1,0);
\draw[<->] (0,-2)--(0,5);
\draw[fill=lightgray, very thin] (0,0)--(-2,0)--(-4,0.5)--(-5.5,2)--(-6,4)--(-6,5)--(0,5)--cycle;
\draw[very thick] (0,5)--(0,0)--(-2,0)--(-4,0.5)(-5.5,2)--(-6,4);
\draw[very thick,dashed] (-4,0.5)--(-5.5,2);
\fill [black,opacity=1] (-2,0) circle (2pt);
\fill [black,opacity=1] (-4,0.5) circle (2pt);
\fill [black,opacity=1] (-5.5,2) circle (2pt);
\draw[->,line width=2pt] (0,0)--node[above right]{\small $\vec{w}''$} (0,1);
\draw[->,line width=2pt] (-5.5,2)--node[above right]{\small $\vec{v}''$} (-5.75,3);
\path (0,5)-- (-0.1,0) node[below right] {\small $p''_{n-1}$}-- (-1.6,0) node[above] {\small $p''_{n-2}$}--(-4.1,0.4) node[above right]{\small $p''_{n-3}$}--(-5.5,2.2) node[right]{\small $p''_{1}$};
\fill [white,opacity=1,draw=black] (0,0) circle (4pt);
\path (-1.6,4.5) node{\small $\p''\setminus\{p_{n-1}''\}$};
\end{tikzpicture}&
\begin{tikzpicture}[scale=0.75,framed]
\draw[<->] (-6,0)--(1,0);
\draw[<->] (0,-2)--(0,5);
\draw[fill=lightgray, very thin] (0,0)--(-2,0)--(-4,0.5)--(-5.5,2)--(-6,4)--(-6,5)--(0,5)--cycle;
\draw[very thick] (0,5)--(0,0)--(-2,0)--(-4,0.5)(-5.5,2)--(-6,4);
\draw[very thick,dashed] (-4,0.5)--(-5.5,2);
\draw[thin,dotted] (-2,0)--(-2,5);
\draw[draw=black,->](-1.3,3.5)--(-1.3,1.5);
\draw[draw=black, ->](-0.7,3.5)--(-0.7,1.5);
\fill [black,opacity=1] (0,0) circle (2pt);
\fill [black,opacity=1] (-2,0) circle (2pt);
\fill [black,opacity=1] (-4,0.5) circle (2pt);
\fill [black,opacity=1] (-5.5,2) circle (2pt);
\path (0,5)-- (-0.1,0) node[below right] {\small $p''_{n-1}$}-- (-1.6,0) node[above] {\small $p''_{n-2}$}--(-4.1,0.4) node[above right]{\small $p''_{n-3}$}--(-5.5,2.2) node[right]{\small $p''_{1}$};
\path (-0.5,4.5) node{\small $\p''$};
\end{tikzpicture}\\
(a)&(b)
\end{tabular}}
\caption{{\small The position of $\p''$ allows us to cover the vertex $p_{n-1}$.}}\label{fig:ori4}
\end{center}
\end{figure}
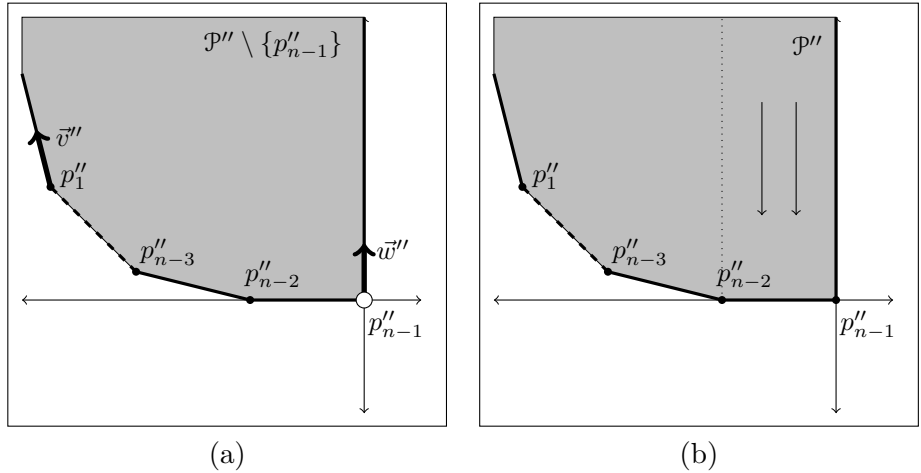

We write $p''_{n-2}=(a''_{n-2},0)$, and apply one last time Lemma~\ref{thm:dob} with $\Ss:=\p''\setminus\{p''_{n-1}\}$, $\mathfrak G_\Ss=\{\ell''_1,\dots,\ell''_{n}\}\supset \{\ell''_1,\dots,\ell''_{n-1}\}= \mathfrak G'_\Ss$, $c:=a''_{n-2}$, $d:=+\infty$, and $\varphi(\x,\y):=\ell''(x,y)=\ell''_1(\x,\y)\cdots \ell''_{n-1}(\x,\y)$ to obtain the polynomial map
$$
f_3:\R^2\to\R^2,\, (x,y)\mapsto(x,(1+\ell''(x,y)(a''_{n-2}-x))^2y),
$$
which gives us
\begin{equation}\label{7}
f_3(\p''\setminus\{p''_{n-1}\})=\p''\setminus\{p''_{n-1}\}\cup(]a_{n-2},0]\times[0,+\infty[)=\p''.
\end{equation}

\vspace{1mm}
\noindent{\bf Conclusion.} By previous equalities \eqref{1} through \eqref{7} we finally have 
\begin{align*}
\tau_1^{-1}\circ\tau_2^{-1}\circ f_3\circ \tau_2\circ f_2\circ \tau_1\circ f_1\circ f_0(\R^2)=\p,
\end{align*}
and the theorem is proved.
\end{proof}

\begin{proof}[Proof of Corollary~\em\ref{cor:pleg}]
First, by \cite[1.4(iv)]{fg1}, there exists a polynomial map $\R^2\rightarrow\R^2$ whose image is the upper open half-plane $\R\times(0,+\infty)$. Since 
$$
\R^3=\R\times\R^2\quad\text{and}\quad\R\times(\R\times(0,+\infty))=\R^2\times(0,+\infty),
$$ 
it is enough to check that ${\rm Int}(\p)$ is the image of the restriction to $\R^2\times(0,+\infty)$ of a
polynomial map $f_1:\R^3\rightarrow\R^2$. We may assume that the two unbounded sides of $\p$ are contained in the lines of equations $x=0$, $y=0$, and ${\rm Int}(\p)\subset\{x>0,y> 0\}$. By Theorem~\ref{thm:pleg}, there exists a polynomial map $f_0:\R^2\rightarrow\R^2$ such that $f_0(\R^2)=\p$. Next, consider the map
$$
\begin{array}{rcl}
f_1:\R^3=\R^2\times\R&\rightarrow&\R^2\\
(x,t)&\mapsto&f_0(x)+t(1,1).
\end{array}
$$
An straightforward computation shows that ${\rm Int}(\p)=f_1(\R^2\times(0,+\infty))$, as wanted.
\end{proof}



\begin{thebibliography}{FGU}


\bibitem[Be]{ber2} M. Berger: Geometry. II. \em Universitext\em. Springer-Verlag, Berlin: 1987.


\bibitem[FG1]{fg1} J.F. Fernando, J.M. Gamboa: Polynomial images of $\R^n$. \em J. Pure Appl. Algebra \em {\bf 179} (2003), no. 3, 241--254.

\bibitem[FG2]{fg2} J.F. Fernando, J.M. Gamboa: Polynomial and regular images of $\R^n$. \em Israel J. Math. \em {\bf 153} (2006), 61--92.

\bibitem[FGU]{fgu1} J.F. Fernando, J.M. Gamboa, C. Ueno: On convex polyhedra as regular images of $\R^n$. \em Proc. London Math. Soc. \em (3) {\bf 103} (2011), 847-878. 

\bibitem[FU1]{fu} J.F. Fernando, C. Ueno: On complements of convex polyhedra as polynomial and regular images of $\R^n$. \em Int. Math. Res. Not. \em doi:10.1093/imrn/rnt112.

\bibitem[FU2]{fu2} J.F. Fernando, C. Ueno: On the set of points at infinity of a polynomial image of $\R^n$. \em Preprint \em RAAG (2012).

\bibitem[J1]{j1} Z. Jelonek: The set of points at which a polynomial map is not proper. \em Ann. Polon. Math. \em {\bf58} (1993), no.1, 259--266.

\bibitem[J2]{j2} Z. Jelonek: A geometry of polynomial transformations of the real plane, \em Bull. Polish Acad. Sci. Math \em {\bf 48} (2000), no. 1, 57-62.


\bibitem[U]{u2} C. Ueno: On convex polygons and their complements as images of regular and polynomial maps of $\R^2$. \em J. Pure Appl. Algebra \em {\bf 216}, no. 11, 2436-2448.


\end{thebibliography}
\end{document}